\documentclass{amsart}

\usepackage{cite}
\usepackage{graphicx}
\usepackage{amsmath,amsthm,amssymb,fullpage,xcolor,enumitem}

\newtheorem{theorem}{Theorem}[section]
\newtheorem{lemma}[theorem]{Lemma}

\theoremstyle{definition}

\theoremstyle{remark}

\numberwithin{equation}{section}

\newcommand{\bD}{\mathbb{D}}
\newcommand{\ds}{\displaystyle}
\newcommand{\lb}{\linebreak}
\newcommand{\abs}[1]{\left\lvert #1 \right\rvert}

\makeatletter
\def\thmhead@plain#1#2#3{%
  \thmname{#1}\thmnumber{\@ifnotempty{#1}{ }\@upn{#2}}%
  \thmnote{{\the\thm@notefont#3}}}
\let\thmhead\thmhead@plain
\makeatother

\begin{document}

\title[{Koebe's theorem for trinomials with fold symmetry}]
{On the Koebe quarter theorem\\ for trinomials with fold symmetry}

\author{Dmitriy Dmitrishin}
\address{Department of Applied Mathematics,
Odessa National Polytechnic University,
Odessa, Ukraine, 65044}
\email{dmitrishin@op.edu.ua}

\author{Daniel Gray}
\address{Department of Mathematical Sciences,
Georgia Southern University,
Statesboro, GA, 30460}
\email{dagray@georgiasouthern.edu}

\author{Alexander Stokolos}
\address{Department of Mathematical Sciences,
Georgia Southern University,
Statesboro, GA, 30460}
\email{astokolos@georgiasouthern.edu}





\begin{abstract}
The Koebe problem for univalent polynomials with real coefficients is fully solved only for trinomials, which means that in this case the Koebe radius and the extremal polynomial (extremizer) have been found. The general case remains open, but conjectures have been formulated. The corresponding conjectures have also been hypothesized for univalent polynomials with real coefficients and $T$-fold rotational symmetry. This paper provides confirmation of these hypotheses for trinomials $z + az^{T + 1} + bz^{2T + 1}$. Namely, the Koebe radius is $r=4\cos^2 \frac{\pi(1+T)}{2+3T}$, and the only extremizer of the Koebe problem is the trinomial
\begin{gather*}
B^{(T)}(z)=z+\frac2{2+3T}\left(-T+(2+2T)\cos\frac{\pi T}{2+3T}\right)z^{1+T}+\\
+\frac1{2+3T}\left(2+T-2T\cos\frac{\pi T}{2+3T}\right)z^{1+2T}.
\end{gather*}

\end{abstract}
\maketitle

\section{Introduction}

This work is motivated by problems of classical geometric complex analysis, which studies various extremal properties
of functions $F(z)$ which are univalent in the central unit disk $\bD=\{z\in\mathbb{C}: \abs{z}<1\}$ and of the form
$$
F(z)=z+\sum_{j=2}^{\infty}a_j z^j,
$$
that is, normalized by $F(0)=0$, $F'(0)=1$. Traditionally, this class is denoted by the symbol $S$ (from the German word Schlicht). 

One of the first fundamental works in this theory was the 1916 paper by Ludwig Bieberbach, in which he proved 
the exact estimate for the second coefficient, namely $\abs{a_2} \le 2$. This estimate immediately implies the famous Koebe 1/4-theorem: $F\in S \Rightarrow F(\bD) \supseteq \bD_r=\{z: \abs{z}<r\}$, $r=1/4$. Further generalizations of Koebe's theorem are related to considering functions from various subclasses of the class $S$. Thus, for bounded in $\bD$
functions, i.e., those satisfying $\abs{f(z)} < M$, the Koebe radius is $r=\left(1+\sqrt{1-1/M}\right)^{-2}$ \cite{Pick}; for
convex in $\bD$ functions, the Koebe radius is $r=1/2$ \cite{Good}.

According to \cite{Good}, the Koebe domain for the family $\Omega\subset S$ is defined as the largest domain that is 
contained in the image $f(\bD)$ for every function $f\in\Omega$. The problem of finding the maximum radius of the 
central disk inscribed into the Koebe domain is called the Koebe problem. This radius is called the Koebe radius. In
\cite{Good, Kocz, Krzyz, Ign}, some examples of finding the Koebe radius for different classes of functions are given. 

For each function $F\in S$, the function $H(z)=\sqrt[T]{F(z^T)}\in S$ ($z\in\bD$, $T=1,2,\ldots$) and maps 
the unit disk to a domain with $T$-fold symmetry. Such functions are called $T$-fold symmetric functions. Using 
the Koebe 1/4-theorem, it is not difficult to obtain the Koebe radius for the $T$-fold symmetric function \cite{Reng}: 
$\ds r=\frac1{2^{2/T}}$. Note that the extremal functions (extremizers) in these problems are $T$-fold symmetric
Koebe functions, which, up to rotation, have the form
$$
K^{(T)}(z)=\frac{z}{(1-z^T)^{2/T}}=z+\sum_{j=2}^N \prod_{s=1}^{j-1}\frac{s+2/T-1}{s} z^{1+(j-1)T}.
$$
The classical Koebe function and the odd Koebe function, respectively, have the representation 
$$
K^{(1)}(z)=\frac{z}{(1-z)^2}=z+\sum_{j=2}^N jz^j, \;
K^{(2)}(z)=\frac{z}{1-z^2}=z+\sum_{j=2}^N z^{2j-1}.
$$

In \cite{Brandt,DDSUniv,DSS}, there was considered the Koebe problem for polynomials of degree $N$ 
with real coefficients:
$$
S^{(N)}=\left\{z+\sum_{j=2}^N a_jz^j:\,z+\sum_{j=2}^N a_jz^j \in S\right\}.
$$
Let $U_k(\cos \vartheta) = \frac{\sin(k + 1)\vartheta}{\sin \vartheta}$, where $\vartheta \in (-\pi,\pi]$, be the Chebyshev polynomials of the second kind ($U_0(t) = 1$, $U_1(t) = 2t$, and $U_2(t) = 4t^2 - 1$) and let 
\begin{equation}\label{Q}
Q_N(z) = z + \sum_{j = 2}^{N}B_jz^j
\end{equation}
where $B_j$ is given by
\begin{eqnarray*}
B_j &=& \frac1{(N+2)\sin\frac{2\pi}{N+2}}
\times\left((N-j+3)\sin\frac{\pi(j+1)}{N+2}-(N-j+1)\sin\frac{\pi(j-1)}{N+2}\right)
\frac{\sin\frac{\pi j}{N+2}}{\sin\frac{\pi}{N+2}} \\
&=& \frac{U'_{N-j+1}\left(\cos\frac{\pi}{N+2}\right)}{U'_N\left(\cos\frac{\pi}{N+2}\right)}
U_{j-1}\left(\cos\frac{\pi}{N+2}\right),\;j=2,\ldots,N.
\end{eqnarray*}
The estimate $\ds r\le\frac1{4}\sec^2\frac{\pi}{N+2}$ on the Koebe radius was obtained by showing that
\begin{equation}\label{pr}
\sup_{p\in S^{(N)}}(p(-1))=Q_N(-1)=-\frac1{4}\sec^2\frac{\pi}{N+2}.
\end{equation} 
Furthermore, it was shown that $Q_N(z)$ is the unique extremizer of problem~\eqref{pr}. The hypothesis that the Koebe 
radius on the class $S^{(N)}$
\begin{equation}\label{rad}
r=\frac1{4}\sec^2\frac{\pi}{N+2}
\end{equation}
was also proposed. This hypothesis was proved in \cite{Ign} for $N = 3$.

In \cite{DTS}, the Koebe problem was considered for the univalent in $\bD$ odd polynomials with real coefficients
$F(z)=z+\sum\limits_{j=2}^n a_jz^{2j-1}$ of degree $N=2n-1$. The Koebe radius is estimated as
\begin{equation}\label{rodd}
r\le\frac1{2}\sec^2\frac{\pi}{N+3},
\end{equation}
it is conjectured that this quantity in the Koebe problem is exact, and the extremizer
\begin{equation}\label{exodd}
F^{(0)}(z)=\sum_{j=1}^n \frac{U'_{N+2-j}\left(\cos\frac{\pi}{N+3}\right)}
{U'_{N+1}\left(\cos\frac{\pi}{N+3}\right)}z^{2j-1}
\end{equation}
is unique.

Moreover, the hypothesis about the exact solution of the Koebe problem for the univalent in $\bD$ polynomials 
with real coefficients $F(z)=z+\sum\limits_{j=2}^n a_jz^{1+T(j-1)}$ of degree $N=1+T(n-1)$ with $T$-fold
symmetry was also proposed there, namely, that the extremizer is unique and is given by the formula
\begin{equation}\label{BT}
B^{(T)}(z)=z+\sum_{j=2}^n b_j\gamma_jz^{1+(j-1)T},
\end{equation}
where
$$
\gamma_j=\prod_{s=1}^{j-1}\frac{\sin\pi\frac{s+2/T-1}{n+2/T}}{\sin\pi\frac{s}{n+2/T}}=
\prod_{s=1}^{j-1}\frac{U_{n-s}\left(\cos\frac{\pi}{n+2/T}\right)}{U_{s-1}\left(\cos\frac{\pi}{n+2/T}\right)}, \;
b_j=\frac{U'_{T(n-j+1)}\left(\cos\frac{\pi}{Tn+2}\right)}{U'_{Tn}\left(\cos\frac{\pi}{Tn+2}\right)},\;j=2,\ldots,n,
$$
and the Koebe radius is $r=\abs{B^{(T)}(e^{i\pi/T})}$.

The image of the unit circle under the mapping by the polynomial $B^{(3)}(z)$ when $N=7$ is given in Fig.~\ref{f1}. 

\begin{figure}[h!]
\includegraphics[scale=0.3]{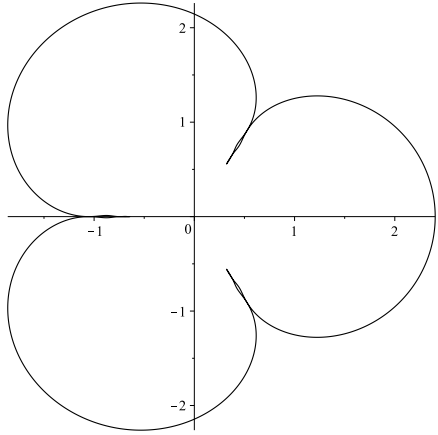}
\caption{Image of the unit circle mapped by the polynomial $B^{(3)}(z)$ $(N=7)$.} \label{f1}
\end{figure} 

Note that when $T=1,2$, polynomial \eqref{BT} coincides with polynomials \eqref{Q} and \eqref{exodd} respectively. 

The aim of this paper is to check the proposed hypotheses for trinomials of the form \eqref{BT}, that is, for the case 
$n=3$ and all $T=1,2,3,\ldots$. Note that even the simplest case $T=1$ was not trivial. The main result is the following: we will show that the trinomial
\begin{equation}\label{Btri}
B^{(T)}(z)=z+a_0z^{1+T}+b_0z^{1+2T},
\end{equation}
where
$$
a_0=\frac{U'_{2T}\left(\cos\frac{\pi}{2+2T}\right)}{U'_{3T}\left(\cos\frac{\pi}{2+2T}\right)}
\frac{\sin\frac{2\pi}{2+3T}}{\sin\frac{\pi T}{2+3T}}=
\frac2{2+3T}\left(-T+(2+2T)\cos\frac{\pi T}{2+3T}\right),
$$
$$
b_0=\frac{U'_{T}\left(\cos\frac{\pi}{2+2T}\right)}{U'_{3T}\left(\cos\frac{\pi}{2+2T}\right)}
\frac{\sin\frac{2\pi}{2+3T} \sin\frac{(2+T)\pi}{2+3T}}{\sin\frac{\pi T}{2+3T} \sin\frac{2\pi T}{2+3T}}
=\frac1{2+3T}\left(2+T-2T\cos\frac{\pi T}{2+3T}\right),
$$
is the extremizer of the Koebe problem, and 
\begin{equation}\label{rad3}
r=4\cos^2\frac{\pi(1+T)}{2+3T}
\end{equation}
is the Koebe radius. 

Note that since $\ds4\cos^2\frac{2\pi}5=\frac1{4}\sec^2\frac{\pi}5$, when $T=1$, quantities \eqref{rad3} and
\eqref{rad} coincide; and for $T=2$, quantities \eqref{rad3} and \eqref{rodd} coincide, 
i.e., $\ds4\cos^2\frac{3\pi}8=\frac1{2}\sec^2\frac{\pi}8$.

\section{Domain of univalence in the coefficient plane for trinomials with fold symmetry}

In \cite{Sch}, there was considered the problem of constructing the domain of univalence for the trinomials
$F(z)=z+az^k+bz^m$ with complex coefficients. In \cite{DSG_arxiv}, the results were refined for the special case 
$k=1+T$, $m=1+2T$ and real $a$, $b$. Let
$$
U_T=\left\{(a,b)\in\mathbb{R}^2:\;z+az^{1+T}+bz^{1+2T}\textit{ is univalent in }\bD\right\}
$$
be the domain of univalence for the trinomial $F(z)=z+az^{1+T}+bz^{1+2T}$ in the plane of the coefficients $a$, $b$.
This domain is bounded by five curves (Fig.~\ref{f2}):
\begin{gather*}
\Gamma_1=\left\{(x,y):\, x=t,\,y=\frac1{1+2T},\right.\\
\left.t\in\left[ -2\frac{1+T}{1+2T}\sin{\frac{\pi}{2+2T}}, 2\frac{1+T}{1+2T}\sin{\frac{\pi}{2+2T}} \right]\right\},
\end{gather*}
\begin{gather*}
\Gamma_2^{+}=\left\{(x,y):\,x=t,\,y=\frac{(1+T)t-1}{1+2T},\,t\in\left[0,\frac4{2+3T}\right]\right\},\\
\Gamma_2^{-}=\left\{(x,y):\,x=-t,\,y=\frac{(1+T)t-1}{1+2T},\,t\in\left[0,\frac4{2+3T}\right]\right\},\\
\Gamma_3^{+}=\left\{(x,y):\,x=\tilde{A}(t),\,y=\tilde{B}(t),\,t\in\left[0,\frac{\pi}{2+2T}\right]\right\},\\
\Gamma_3^{-}=\left\{(x,y):\,x=-\tilde{A}(t),\,y=\tilde{B}(t),\,t\in\left[0,\frac{\pi}{2+2T}\right]\right\},
\end{gather*}
where 
$$
\tilde{A}(t)=\frac{2T\sin{(2+2T)t}-(2+2T)\sin{2Tt}}{T\sin{(2+3T)t}-(2+3T)\sin{Tt}},\;
\tilde{B}(t)=\frac{T\sin{(2+T)t}-(2+T)\sin{Tt}}{T\sin{(2+3T)t}-(2+3T)\sin{Tt}}.
$$

Let us note that the domain $U_1$ is defined in \cite{Bran, Cow}; when $T=1$, the curves $\Gamma_3^{+}$, 
$\Gamma_3^{-}$ are arcs of the ellipse $a^2+4(1/2-b)^2=1$. Notice also that for $T=2$, the curves 
$\Gamma_3^{+}$, $\Gamma_3^{-}$ are arcs of the ellipses $(a\pm b)^2-4b(1-b)=0$.

\begin{figure}[h!]
\includegraphics[scale=0.55]{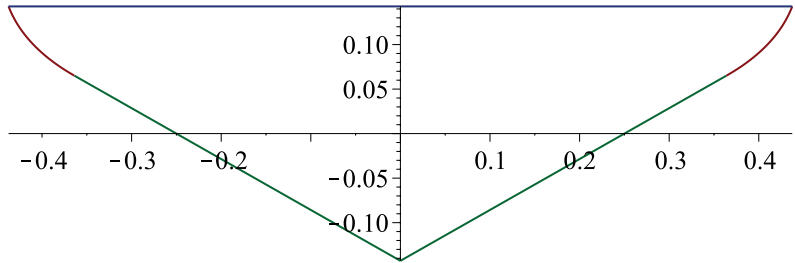}
\caption{Boundary of the domain $U_3$: the curves $\Gamma_1$ (blue), $\Gamma_2^{+}$, 
$\Gamma_2^{-}$ (green), $\Gamma_3^{+}$, $\Gamma_3^{-}$ (red).} \label{f2}
\end{figure}

Note that the domain $U_T$, considered in the coefficient plane $(a,b)$, has axial symmetry about the line $Ob$,
therefore we can restrict ourselves, without loss of generality, to the case $a\ge0$.

In what follows, it turns out to be convenient to make the substitution $\alpha=\frac1{1+T}$ 
$\left(\text{or }T=\frac1\alpha-1\right)$, and then $\alpha\in(0,1/2]$, since $T\ge1$. The boundaries of the univalence
domain after the substitution is given by the parametric equations:
\begin{gather*}
\Gamma_1=\left\{(x,y):\,x=t,\,y=\frac{\alpha}{2-\alpha},\,
t\in\left[-\frac2{2-\alpha}\sin\frac{\pi\alpha}2,\frac2{2-\alpha}\sin\frac{\pi\alpha}2\right]\right\},\\
\Gamma_2^{+}=\left\{(x,y):\,x=t,\,y=\frac{t-\alpha}{2-\alpha},\,t\in\left[0,\frac{4\alpha}{3-\alpha}\right]\right\},\\
\Gamma_2^{-}=\left\{(x,y):\,x=-t,\,y=\frac{t-\alpha}{2-\alpha},\,t\in\left[0,\frac{4\alpha}{3-\alpha}\right]\right\},\\
\Gamma_3^{+}=\left\{(x,y):\,x=A(t),\,y=B(t),\,t\in\left[0,\frac{\pi}2\right]\right\},\\
\Gamma_3^{-}=\left\{(x,y):\,x=-A(t),\,y=B(t),\,t\in\left[0,\frac{\pi}2\right]\right\},
\end{gather*} 
where $\ds A(t)=\frac{U(t)}{W(t)}$, $\ds B(t)=\frac{V(t)}{W(t)}$,
\begin{gather*}
U(t)=2\sin2(1-\alpha)t-2(1-\alpha)\sin2t,\\
V(t)=(1+\alpha)\sin(1-\alpha)t-(1-\alpha)\sin(1+\alpha)t,\\
W(t)=(3-\alpha)\sin(1-\alpha)t-(1-\alpha)\sin(3-\alpha)t.
\end{gather*}

With the new parameterization, the conjecture about the solution of the Koebe problem takes the next form. The Koebe 
radius for $T$-fold symmetric univalent trinomials is
$$
r=4\cos^2\frac{\pi}{3-\alpha};
$$
accordingly, the coefficients of the extremal trinomial are:
\begin{gather*}
a_0=-\frac{2-2\alpha}{3-\alpha}+\frac4{3-\alpha}\cos\frac{\pi(1-\alpha)}{3-\alpha}=
-2+\frac8{3-\alpha}\sin^2\frac{\pi}{3-\alpha},\\
b_0=\frac{1+\alpha}{3-\alpha}-\frac{2-2\alpha}{3-\alpha}\cos\frac{\pi(1-\alpha)}{3-\alpha}=
1-\frac{4(1-\alpha)}{3-\alpha}\sin^2\frac{\pi}{3-\alpha}.
\end{gather*}

This new form of defining the boundaries of the univalence domain is especially convenient when analyzing the 
behavior of necessary objective functions on the curve $\Gamma_3^{+}$, which is key to solving the entire Koebe 
extremal problem. 

\section{Main result}

\subsection*{\indent 3.1 Objective function} 

Let us introduce the function equal to the squared distance from zero to the boundary of the image of the central
unit disk under the mapping $F(z)=z+az^{1+T}+bz^{1+2T}$:
$$
\Phi(a,b,\varphi)=\left|F(e^{i\varphi})\right|^2=1+a^2+b^2+2a(1+b)\cos T\varphi+2b\cos2T\varphi.
$$
The Koebe problem reduces to finding the minimum of this function on the set $(a,b,\varphi)\in U_T\times[0,\pi/T]$.
Denote this minimum by $\Phi^*$.

\subsection*{\indent 3.2 Extremum on the boundary} 

First, let us find the partial derivatives
$$
\frac{\partial\Phi}{\partial a}=2a+2(1+b)\cos T\varphi,\quad
\frac{\partial\Phi}{\partial b}=2a\cos T\varphi+2b+2\cos2T\varphi
$$ 
and set them equal to zero. This gives us $a=-2\cos T\varphi$, $b=1$. Since $(-2\cos T\varphi,1)\notin U_T$, this 
means that the extremal polynomial corresponds exactly to a boundary point of the domain $U_T$. Due to the
symmetry of this region, it suffices to consider the problem of minimizing the function $\Phi(a,b,\varphi)$ only on the
curves $\Gamma_1$, $\Gamma_2^{+}$, $\Gamma_3^{+}$. 

\subsection*{\indent 3.3 Main and special directions} 

Let us call the direction $e^{i\frac{\pi}T}$ the main one. The direction $e^{i\hat{\varphi}}$, 
$0\le\hat{\varphi}<\pi/T$, we will call special if 
$\Phi(a,b,\hat\varphi)=\min\limits_{\varphi\in[0,\pi/T)}\Phi(a,b,\varphi)<\Phi(a,b,\pi/T)$, i.e., contraction of the image 
of the unit disk is stronger in the direction $e^{i\hat\varphi}$ than in the main direction.

We find $\ds\frac{\partial\Phi}{\partial\varphi}=-2aT(1+b)\sin T\varphi-4bT\sin2T\varphi$, set it equal to zero, and 
obtain $\ds\frac{a(1+b)}{4b}=-\cos T\varphi$. This implies that the special direction does not always exist since the
condition $\ds\left|\frac{a(1+b)}{4b}\right|\le1$ must be satisfied; moreover, it also follows that if the special direction
exists, then it is unique. We will see later that the special direction exists for all $(a,b)\in\Gamma_1$, and that the special direction exists only for some 
$(a,b)\in\Gamma_3^{+}$.

\subsection*{\indent 3.4 Studying contraction of the image of the unit disk in the main direction}

For our functions $F(z) = z + az^{1 + T} + bz^{1 + 2T}$, we have $F(e^{i\pi/T}) = e^{i\pi/T}(1 - a + b)$. Hence, we set $h(a,b)=-a+b$, and
$$
H(a,b)=\Phi(a,b,\pi/T)=(1+h(a,b))^2.
$$
It is clear that 
$\min\limits_{(a,b)\in U_T}H(a,b)\ge\Phi^*$. On the curve $\Gamma_1$, 
$\ds h(a,b)=-t+\frac{\alpha}{2-\alpha}$; on $\Gamma_2^+$, 
$\ds h(a,b)=-t+\frac{t-\alpha}{2-\alpha}=-\frac{(1-\alpha)t+\alpha}{2-\alpha}$. Obviously, the 
function $h(a,b)$ decreases on the curves $\Gamma_1$, $\Gamma_2^+$.

Let us study $h(a,b)$ on the curve $\Gamma_3^+$, i.e. let us examine the behavior of the function 
$h(A(t),B(t))=-A(t)+B(t)$.

According to Lemma~\ref{l1}, 
$$
A'(t)=\frac{\sin(2-\alpha)t}{\sin t}B'(t)>0,\; t\in\left(0,\frac{\pi}2\right).
$$
From this,
$$
\frac{d}{dt}h(A(t),B(t))=-A'(t)+B'(t)=\frac{B'(t)}{\sin t}(\sin t-\sin(2-\alpha)t)
=-\frac{2B'(t)}{\sin t}\cos\frac{3-\alpha}2 t\sin\frac{1-\alpha}2 t.
$$
Therefore, the function $h(A(t),B(t))$ decreases when $t\in\left(0,\frac{\pi}{3-\alpha}\right)$, increases when 
$t\in\left(\frac{\pi}{3-\alpha},\frac{\pi}2\right)$, and has a local minimum at ${t=t^*=\frac{\pi}{3-\alpha}}$. This
means that the contraction in the main direction is maximized when the trinomial coefficients take the values 
$a=A(t^*)=a_0$, $b=B(t^*)=b_0$.

Thus, 
$$
\min_{(a,b)\in U_T}H(a,b)=H(a_0,b_0)=\left[4\cos^2\frac{\pi}{3-\alpha}\right]^2=
\left[4\cos^2\frac{\pi(1+T)}{2+3T}\right]^2.
$$
The minimum is attained at a single point which corresponds to polynomial \eqref{Btri} and is equal to 
quantity \eqref{rad3} squared. 

\begin{figure}[h!]
\includegraphics[scale=0.57]{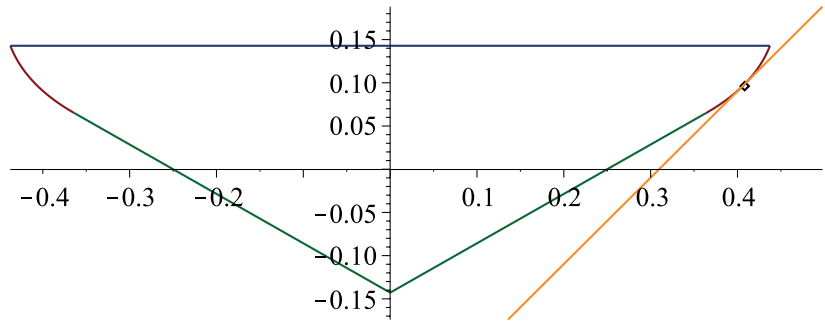}
\caption{Tangent to $\Gamma_3^+$ at the point $(a_0,b_0)$ with the equation 
$-x+y=-1+4\cos^2\frac{\pi(1+T)}{2+3T}$ $(T=3)$.} \label{f3}
\end{figure}

\subsection*{\indent 3.5 Property of the contraction quantity of the disk image in the special direction} 

Let us represent the function $\Phi(a,b,\varphi)$ in the form
$$
\Phi(a,b,\varphi)=H(a,b)+16b\cos^2\frac{T\varphi}2 \left(\frac{a(1+b)}{4b}-\sin^2\frac{T\varphi}2\right).
$$
Denote $\ds\mu(a,b)=\frac{a(1+b)}{4b}$. The behavior of this function plays a key role in determining 
the properties of the objective function $\Phi(a,b,\varphi)$.  

\subsubsection*{\indent \bf 3.5.1 Curve $\Gamma_2^+$}
\indent

On the curve $\Gamma_2^+$: for negative $b$, that is, when $t\in(0,\alpha)$, 
$$
\Phi(a,b,\varphi)=
H(a,b)+16|b|\cos^2\frac{T\varphi}2\left(\frac{a(1-|b|)}{4|b|}+\sin^2\frac{T\varphi}2\right)>H(a,b).
$$ 
For positive $b$, i.e., when $t\in\left(\alpha,\frac{4\alpha}{3-\alpha}\right]$, we have 
$$
\mu(a,b)=\frac{t\left(1+\frac{t-\alpha}{2-\alpha}\right)}{4\cdot\frac{t-\alpha}{2-\alpha}}=
\frac{t^2+2t(1-\alpha)}{4(t-\alpha)}.
$$
This function decreases in $t$, hence
$$
\mu(a,b)\ge\mu\left(\frac{4\alpha}{3-\alpha},\frac1{2-\alpha}\left(\frac{4\alpha}{3-\alpha}-\alpha\right)\right)
=\frac{2(\alpha^2-2\alpha+3)}{(3-\alpha)(1+\alpha)}\ge
\left.\frac{2(\alpha^2-2\alpha+3)}{(3-\alpha)(1+\alpha)}\right|_{\alpha=\frac12}=\frac65>1.
$$
Therefore, on $\Gamma_2^+$, there holds $\Phi(a,b,\varphi)>H(a,b)>H(a_0,b_0)$ for all 
$t\in\left(\alpha,\frac{4\alpha}{3-\alpha}\right]$. This means that there is no special direction 
on the curve $\Gamma_2^+$.

\subsubsection*{\indent \bf 3.5.2 Curve $\Gamma_1$}
\indent

Let us investigate the behavior of the function $\Phi(a,b,\hat\varphi)$, where $\hat\varphi$ determines the special 
direction on the curve $\Gamma_1$. On this curve, 
$$
a=t,\;b=\frac{\alpha}{2-\alpha},\;\mu(a,b)=\frac{t}{2\alpha},\;
t\in\left[-\frac2{2-\alpha}\sin\frac{\pi\alpha}2,\frac2{2-\alpha}\sin\frac{\pi\alpha}2\right].
$$
Recall that the special direction occurs when $\mu(a,b)=-\cos T\hat\varphi$. Hence, $\ds\cos^2\frac{T\hat\varphi}2=\frac{1-\mu(a,b)}2$,  
$\ds\sin^2\frac{T\hat\varphi}2=\frac{1+\mu(a,b)}2$. Substituting this into $\Phi(a,b,\hat\varphi)$ yields 
\begin{gather*}
\Phi(a,b,\hat\varphi)=
H(a,b)+16b\cos^2\frac{T\hat\varphi}2\left(\frac{a(1+b)}{4b}-\sin^2\frac{T\hat\varphi}2\right)=\\
=(1-a+b)^2-4b\left(1-\frac{a(1+b)}{4b}\right)^2=(1-b)^2\left(1-\frac{a^2}{4b}\right).
\end{gather*} 
Then $\ds\Phi(a,b,\hat\varphi)=\frac{4(1-\alpha)^2}{(2-\alpha)^2}\left(1-\frac{2-\alpha}{4\alpha}t^2\right),$ and since this is clearly decreasing in $t$ for $\alpha \in (0,\frac{1}{2}]$ we obtain that 
$$
\Phi(a,b,\varphi)\ge\frac{4(1-\alpha)^2}{(2-\alpha)^2}\left(1-\frac1{\alpha(2-\alpha)}\sin^2\frac{\pi\alpha}2\right).
$$
Thus, the minimum of the function $\Phi(a,b,\hat\varphi)$ is attained at such values of the coefficients that 
correspond to the common point of the curves $\Gamma_1$ and $\Gamma_3^+$. This point determines the coefficients
of the generalized Suffridge trinomial \cite{DSS} 
$$
S(z)=z+\frac2{2-\alpha}\sin\frac{\pi\alpha}2 z^{1+T}+\frac{\alpha}{2-\alpha}z^{1+2T}.
$$
This minimum equals
$$
R^2=4\left(\frac{1-\alpha}{2-\alpha}\right)^2\left(1-\frac1{\alpha(2-\alpha)}\sin^2\frac{\pi\alpha}2\right).
$$
It follows from Lemma~\ref{l2} that on the curve $\Gamma_1$, this minimum is greater than 
$\left[4\cos^2\frac{\pi}{3-\alpha}\right]^2$, which means that the minimum distance from the boundary of the image
of the unit disk to zero, under the mapping by the generalized Suffridge polynomial, exceeds quantity \eqref{Btri}:
$$
\frac{2(1-\alpha)}{2-\alpha}\sqrt{1-\frac1{\alpha(2-\alpha)}\sin^2\frac{\pi\alpha}2}>4\cos^2\frac{\pi}{3-\alpha}.
$$

Fig.~\ref{f4} shows the image of the boundary of the central unit disk under the mappings 
$F(z)=z+a_0 z^{1+T}+b_0 z^{1+2T}$ and $S(z)$; the corresponding disks are provided. It can be seen that the
special direction exists for the polynomial $S(z)$. Note that there also exists the special direction for all other points on
the curve $\Gamma_1$ (since $\mu(a,b)<1$). 

\begin{figure}[h!]
\centering
\includegraphics[scale=0.27]{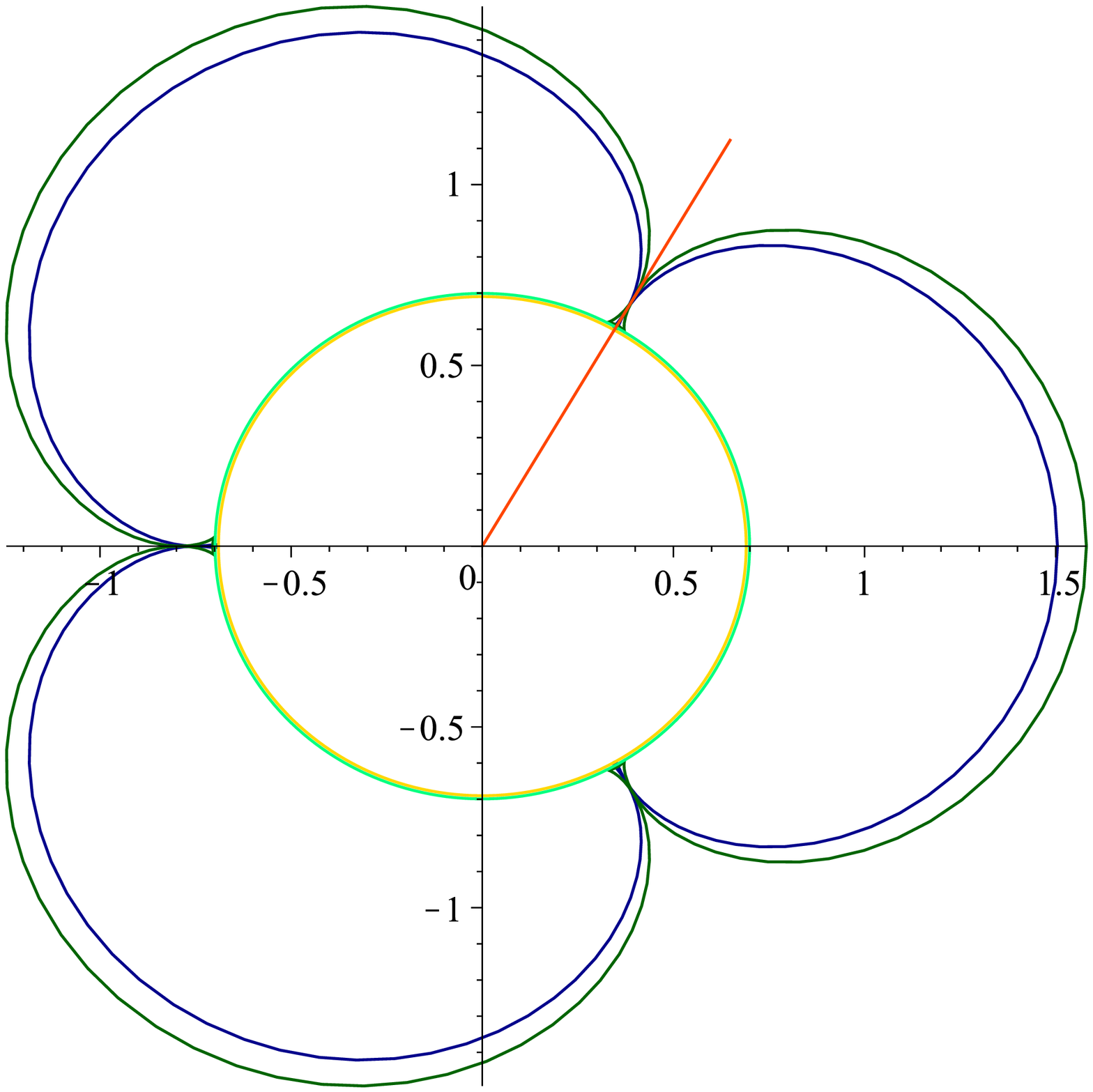}
\hspace{1cm}
\includegraphics[scale=0.27]{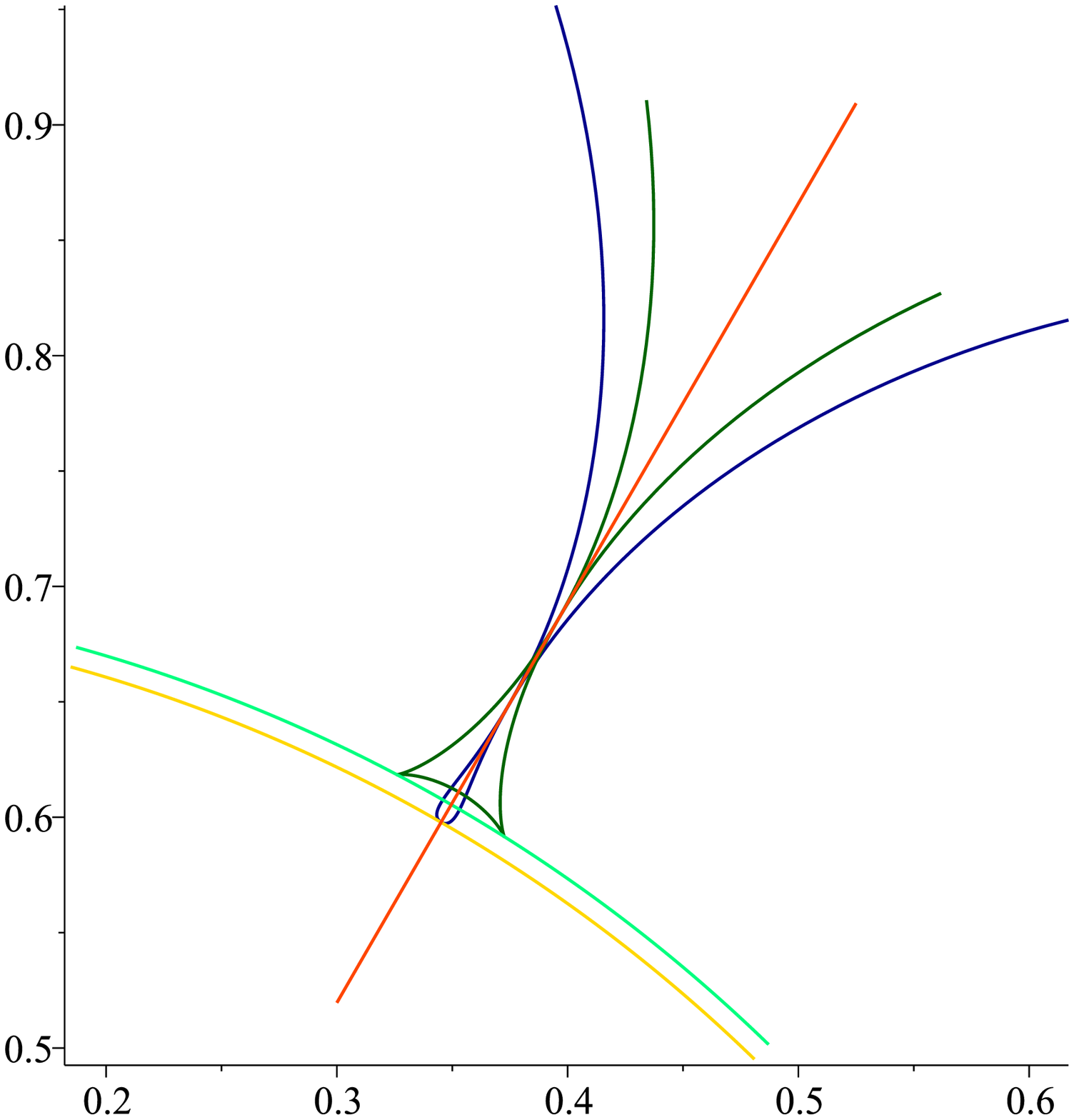}
\caption{Image of the boundary of the central unit disk under the mappings $F(z)=z+a_0 z^{1+T}+b_0 z^{1+2T}$
(blue), $S(z)$ (green); central disks of radius ${r=4\cos^2\frac{\pi}{3-\alpha}}$ (yellow), radius 
${R=\frac{2(1-\alpha)}{2-\alpha}\sqrt{1-\frac1{\alpha(2-\alpha)}\sin^2\frac{\pi\alpha}2}}$ (light green) ($T=3$,
$\alpha=1/4$).} \label{f4}
\end{figure}

\subsubsection*{\indent \bf 3.5.3 Curve $\Gamma_3^+$}
\indent

According to Lemma~\ref{l1}, the functions $A(t)$, $B(t)$ increase on the curve $\Gamma_3^+$. It follows from
Lemma~\ref{l3} that the function $\mu(a,b)=\frac{a(1+b)}{4b}$ decreases on this curve. Moreover, $\mu(a,b)>1$ at 
$t=t^*=\frac{\pi}{3-\alpha}$ (Lemma~\ref{l4}). This means that the special direction will not exist for the parameters $a = A(t)$, $b = B(t)$, $t \in [0,\tilde{t})$ and that the special direction will exist for the parameters 
$a=A(t)$, $b=B(t)$, $t\in(\tilde{t},\pi/2]$, where $\tilde{t}$ is determined by the condition 
$\ds\mu(A(t),B(t)) = \frac{A(\tilde{t})(1+B(\tilde{t}))}{4B(\tilde{t})}=1$ with $\tilde{t}>t^*$. We need to show that the distance
in the special direction is greater than $4\cos^2\frac{\pi}{3-\alpha}$, or
\begin{equation}\label{Phig}
\Phi(a,b,\hat\varphi)=(1-b)^2\left(1-\frac{a^2}{4b}\right)>\left[4\cos^2\frac{\pi}{3-\alpha}\right]^2
\end{equation}
when $a=A(t)$, $b=B(t)$, $t\in(\tilde{t},\pi/2]$. For $t\le\tilde{t}$, we have already shown that $\Phi(a_0,b_0,e^{i\pi/T})$ attains the minimum value.

In what follows, we will omit the excluded variable $\hat\varphi$; that is, the function $\Phi(a,b,\hat\varphi)$ will be 
denoted by $\Phi(a,b)$, which is given by
$$
\Phi(a,b)=(1-b)^2\left(1-\frac{a^2}{4b}\right).
$$  

Examine the possible cases. 

\begin{itemize}
\item First, consider $T=1$ $\left(\alpha=\frac12\right)$. In this case, on the curve $\Gamma_3^+$, the parameters $a$ and $b$ are 
related by the relationship $a^2+4(b-1/2)^2=1$. Then 
$$
\Phi(a,b)=(1-b)^2\left(1-\frac{1-4(b-1/2)^2}{4b}\right)=b(1-b)^2.
$$
This function is increasing when $b\le\frac13$ (in the domain $U_T$, $b\le\frac1{2T+1}$ is satisfied for all $T$),
hence, the inequality holds. 

\item Next, consider $T=2$ $\left(\alpha=\frac13\right).$ In this case, on the curve $\Gamma_3^+$, the parameters $a$ and $b$ are 
related by the relationship $(a+b)^2-4b(1-b)=0$. Then 
$$
\Phi(a,b)=(1-b)^2\left(\frac34 b+\sqrt{b(1-b)}\right).
$$
This function increases when $b\le\frac15$ (in the domain $U_2$, $b\le\frac15$). Indeed, the derivative
$$
\frac{d}{db}\left[(1-b)^2\left(\frac34 b+\sqrt{b(1-b)}\right)\right]
=\frac14 (1-b)\left(3-9b-12\sqrt{b(1-b)}+2\sqrt{\frac1b-1}\right)
$$
has three critical points: $0$, $0.21\ldots$, $1$, and in a neighborhood of zero it is positive. This implies the validity 
of the inequality.  

\item Now, we handle $T\ge7$ $\left(\alpha\in\left(0,\frac18\right]\right)$. In this case, the parameters $a$ and $b$ are not related
by a simple relationship. Thus, showing that the function $\Phi(a,b)$ increases on $\Gamma_3^+$ turns out to be difficult as we would need to build fine estimates for the function $\mu(a,b)$; a simple bound from above
by 1 and decreasing are not enough. However, inequality \eqref{Phig} can be proved more easily. 

Find
$$
\frac{\partial}{\partial a}\Phi(a,b)=-\frac{a(1-b)^2}{2b}<0,\quad
\frac{\partial}{\partial b}\Phi(a,b)=\frac{1-b}{4b^2}(a^2-8b^2+a^2b).
$$  
Lemma~\ref{l5} implies that $\frac{\partial}{\partial b}\Phi(a,b)>0$. It follows from the inequality $\mu(a,b)\le1$ 
for $t\in(\tilde{t},\pi/2]$ that $b\ge\frac{a}{4-a}$. Since the function $\Phi(a,b)$ increases in $b$, then
$$
\Phi(a,b)\ge\Phi\left(a,\frac{a}{4-a}\right)=\frac{(2-a)^4}{(4-a)^2}.
$$
The function $\ds\frac{(2-a)^2}{4-a}$ is positive and decreases in $a$, hence
$$
\Phi(a,b)\ge\frac{(2-A(\pi/2))^4}{(4-A(\pi/2))^2}=
\left(\frac{2\big(2-\alpha-\sin(\pi\alpha)/2\big)^2}{(2-\alpha)(4-2\alpha-\sin(\pi\alpha)/2)}\right)^2.
$$
From Lemma~\ref{l6}, we obtain that 
$$
\frac{2(2-\alpha-\sin(\pi\alpha)/2)}{(2-\alpha)(4-2\alpha-\sin(\pi\alpha)/2)}>4\cos^2\frac{\pi}{3-\alpha}
$$ 
for $\alpha\in(0,1/8]$, or $\Phi(a,b)>\left[4\cos^2\frac{\pi}{3-\alpha}\right]^2$.

\item The remaining cases are $T\in\{3,4,5,6\}$ $\left(\alpha\in\left\{\frac14,\frac15,\frac16,\frac17\right\}\right)$. The strategy employed for each of these cases is identical. An algorithm for verifying inequality \eqref{Phig} follows:
\begin{enumerate}[label={\bf \arabic*)}] \item Consider the inequality $\ds\frac{(2-A(t))^2}{4-A(t)}\ge4\cos^2\frac{\pi}{3-\alpha}$. Compute 
a $t_0$ for which the inequality holds and for which the value of the difference 
$$
\delta_0=\frac{(2-A(t))^2}{4-A(t)}-4\cos^2\frac{\pi}{3-\alpha}
$$ 
is sufficiently small. This means that inequality \eqref{Phig} is satisfied on the interval $t\in(\tilde{t},t_0]$) since $\frac{(2 - A(t))^2}{4 - A(t)}$ decreases as $A(t)$ increases (and $A(t)$ increases as $t$ increases by Lemma~\ref{l1}).
\item Calculate the quantity $\ds\mu_0=\frac{A(t_0)(1+B(t_0))}{4B(t_0)}$. Since $\mu(a,b)$ is decreasing in $t$ this implies $\mu(a,b)<\mu_0$ when 
$t\in(t_0,\pi/2]$. Hence, the inequality 
$$
b>\frac{a/\mu_0}{4-a/\mu_0}=\frac{a}{4\mu_0-a}
$$ 
holds.
\item Construct the function 
$$
\Phi\left(a,\frac{a}{4\mu_0-a}\right)=\frac{(2\mu_0-a)^2}{(4\mu_0-a)^2}(a^2-4\mu_0a+4).
$$
Since $\frac{\partial \Phi}{\partial b} > 0$ by Lemma~\ref{l5},  it follows that $\ds\Phi(a,b)\ge\frac{(2\mu_0-a)^2}{(4\mu_0-a)^2}(a^2-4\mu_0a+4)$. Now, we must verify that the function $\frac{(2\mu_0-a)^2}{(4\mu_0-a)^2}(a^2-4\mu_0a+4)$ is decreasing in $a$. It follows
from the inequality $b\le B(\pi/2)=\frac{\alpha}{2-\alpha}$ that $\frac{1+b}{4b}\ge\frac1{2\alpha}\ge2$. This implies $\mu_0>\mu(a,b)\ge2a$. Thus,
\begin{gather*}
\frac{d}{da}\left[\frac{(2\mu_0-a)^2}{(4\mu_0-a)^2}(a^2-4\mu_0 a+4)\right]=\\
=2\frac{2\mu_0-a}{(4\mu_0-a)^3}\left(-16\mu_0^2(\mu_0-7/4a)-a^2(10\mu_0-a) - 8\mu_0\right)<0.
\end{gather*}
\item Check the inequality
\begin{displaymath}
\Delta=\frac{(2\mu_0-A(\pi/2))^2}{(4\mu_0-A(\pi/2))^2}\left(A^2(\pi/2)-4\mu_0A(\pi/2)+4\right)-\left[4\cos^2\frac{\pi}{3-\alpha}\right]^2>0,
\end{displaymath}
where $\ds A(\pi/2)=\frac2{2-\alpha}\sin\frac{\pi\alpha}2$. Therefore, inequality \eqref{Phig} is satisfied for all 
$t\in(\tilde{t},\pi/2]$. 
\end{enumerate}
It remains to apply the algorithm for each of the four considered cases, i.e., for 
$\alpha\in\left\{\frac14,\frac15,\frac16,\frac17\right\}$.
\begin{gather*}
\alpha=\frac14:\;t_0=1.46,\,\delta_0<2.3\cdot10^{-5},\,\mu_0=0.950\ldots,\,\Delta=0.0019\ldots>0;\\
\alpha=\frac15:\;t_0=1.49,\,\delta_0<5.4\cdot10^{-5},\,\mu_0=0.921\ldots,\,\Delta=0.0070\ldots>0;\\
\alpha=\frac16:\;t_0=1.52,\,\delta_0<1.9\cdot10^{-5},\,\mu_0=0.890\ldots,\,\Delta=0.010\ldots>0;\\
\alpha=\frac17:\;t_0=1.55,\,\delta_0<2.2\cdot10^{-6},\,\mu_0=0.857\ldots,\,\Delta=0.012\ldots>0.
\end{gather*}
Thus, it is shown that the distance in the special direction is always greater than $\ds r=4\cos^2\frac{\pi}{3-\alpha}$.
\end{itemize}

\section{Main Result}

We restate the main result for convenience here.

\begin{theorem}\label{main}
Let $B^{(T)}(z) = z + a_0z^{1 + T} + b_0z^{1+2T}$ where
$$
a_0 = \frac{2}{2 + 3T}\left(-T + (2 + 2T)\cos\left(\frac{\pi T}{2 + 3T}\right)\right), \;
b_0 =\frac{1}{2 + 3T}\left(2 + T - 2T\cos\left(\frac{\pi T}{2 + 3T}\right)\right).
$$
Then, $\ds\min_{F \in \mathcal{S}^T} \min_{z \in \bar{\mathbb{D}}} \{\abs{F(z)}\} = \abs{B^{(T)}(e^{i\pi/T})} = 4\cos^2\left(\frac{\pi(1 + T)}{2 + 3T}\right)$.
\end{theorem}

The proof can be summarized as follows.
\begin{itemize}
\item We constructed the function $\Phi(a,b,\varphi) = \abs{F(e^{i\varphi})}^2$ where $F(z) = z + az^{1 + T} + bz^{1 + 2T}$. Finding the Koebe radius over $\mathcal{S}^T$ is equivalent to minimizing this quantity.
\item We showed that the minimizer of $\Phi(a,b,\varphi)$ must be on the boundary of the univalence domain. Due to symmetry, we restricted ourselves to points on the boundary with $a > 0$.
\item On $\Gamma_1$, we showed that there was a special direction, and that $\Phi(a,b,\varphi) > 4\cos^2\left(\frac{\pi}{3- \alpha}\right)$ in both the main direction and the special direction for all $(a,b)$ on $\Gamma_1$.
\item On $\Gamma_2^+$, we showed that there was no special direction, and that $\Phi(a,b,\varphi) > 4\cos^2\left(\frac{\pi}{3- \alpha}\right)$ in the main direction for all $(a,b)$ on $\Gamma_2^+$.
\item On $\Gamma_3^+$, we showed that for some $(a,b)$ there was a special direction. In particular, for $(a_0,b_0)$ there exists a special direction, $\hat{\varphi}$. Then, we demonstrated $\ds \Phi(a,b,\varphi) \geq \Phi(a_0,b_0,\hat{\varphi}) = 4\cos^2\left(\frac{\pi}{3- \alpha}\right)$ for all $(a,b)$ on $\Gamma_3^+$, where $\Phi$ is minimized at the point $(a_0,b_0)$ on $\Gamma_3^+$ in the main direction.
\end{itemize}
Thus, the Koebe radius is given by $\ds
4\cos^2\left(\frac{\pi}{3- \alpha}\right) = 4\cos^2\left(\frac{(1 + T)\pi}{2 + 3T}\right)$, and is achieved at the point $(a_0,b_0)$ on $\Gamma_3^+$ in the main direction.



\section{Appendix. Auxiliary results}

\begin{lemma}[\cite{DSS}]\label{l1}
Let $A(t)$ and $B(t)$ be defined as in the boundary $\Gamma_3^+$ of the domain of univalence $U_T$. Then,
$$
A'(t)=\frac{2(1-\alpha)G(t,\alpha)}{\alpha^2(W(t))^2}\sin(2-\alpha)t, \;
B'(t)=\frac{2(1-\alpha)G(t,\alpha)}{\alpha^2(W(t))^2}\sin t,
$$
where
$$
G(t,\alpha)=(1+\alpha)\cos(3-2\alpha)t+(3-\alpha)\cos(1-2\alpha)t-(1-\alpha)^2\cos3t-(1+\alpha)(3-\alpha)\cos t,
$$
and $G(t,\alpha)>0$ when $t\in(0,\pi/2)$, $\alpha\in(0,1/2]$. 
\end{lemma}

\begin{lemma}\label{l2}
The inequality
\begin{equation}\label{inl2}
\frac{4(1-\alpha)^2}{(2-\alpha)^2}\left(1-\frac1{\alpha(2-\alpha)}\sin^2\frac{\pi\alpha}2\right)>
\left[4\cos^2\frac{\pi}{3-\alpha}\right]^2
\end{equation}
 holds true for $\alpha\in(0,1/2]$.
\end{lemma}
\begin{proof}
Write $\ds\cos\frac{\pi}{3-\alpha}=\cos\left(\frac{\pi}3+\frac{\pi}3\frac{\alpha}{3-\alpha}\right)=
\frac12\cos\frac{\pi\alpha}{3(3-\alpha)}-\frac{\sqrt3}2\sin\frac{\pi\alpha}{3(3-\alpha)}$, whence
\begin{gather*}
\cos^2\frac{\pi}{3-\alpha}=\frac14+\frac12\sin^2\frac{\pi\alpha}{3(3-\alpha)}-
\frac{\sqrt3}4\sin\frac{2\pi\alpha}{3(3-\alpha)}<\\
<\frac14+\frac12\left(\frac{\pi\alpha}{3(3-\alpha)}\right)^2-\frac{\sqrt3}4
\left(\frac{2\pi\alpha}{3(3-\alpha)}-\frac16\left(\frac{2\pi\alpha}{3(3-\alpha)}\right)^3\right)=\\
=\frac14-\frac{\sqrt3}2\frac{\pi\alpha}{3(3-\alpha)}+\frac12\left(\frac{\pi\alpha}{3(3-\alpha)}\right)^2+
\frac{\sqrt3}3\left(\frac{\pi\alpha}{3(3-\alpha)}\right)^3=G_1(\alpha).
\end{gather*}
Next, 
$$
\sin^2\frac{\pi\alpha}2=\frac12(1-\cos\pi\alpha)<\frac{(\pi\alpha)^2}4-\frac{(\pi\alpha)^4}{48}+
\frac{(\pi\alpha)^6}{1440}=G_2(\alpha).
$$
Let us represent inequality \eqref{inl2} in the equivalent form
$$
\alpha(2-\alpha)\left(1-\frac{4(2-\alpha)^2}{(1-\alpha)^2}\cos^4\frac{\pi}{3-\alpha}\right)>\sin^2\frac{\pi\alpha}2.
$$
This inequality will certainly hold if 
\begin{equation}\label{col2}
\alpha(2-\alpha)\left(1-\frac{4(2-\alpha)^2}{(1-\alpha)^2}(G_1(\alpha))^2\right)>G_2(\alpha).
\end{equation}
Inequality \eqref{col2} is transformed into the form $\frac{10^{-7}}{1440\cdot3^5}\frac{\alpha^2}{(3-\alpha)^6(1-\alpha)^2}R(\alpha)>0,
$ where $R(\alpha)=\sum\limits_{j=0}^{12}b_j\alpha^j$ is a polynomial of degree 12, and the coefficients
$b_0$, $b_2$, $b_4$, $b_6$, $b_7$, $b_9$, $b_{11}$ are positive, and $b_1$, $b_3$, $b_5$, $b_8$, $b_{10}$,
$b_{12}$ are negative. These coefficients can be calculated with the required degree of accuracy and then estimated
by replacing them by smaller ones, for which we may leave, for instance, two decimal places. 
As a result, we obtain
\begin{gather*}
\hat{R}(\alpha)=9.42-72.60\alpha+240.39\alpha^2-444.26\alpha^3+490.13\alpha^4
-307.97\alpha^5+71.46\alpha^6+ \\+43.50\alpha^7-44.53\alpha^8
+17.92\alpha^9-3.95\alpha^{10}+0.46\alpha^{11}-0.03\alpha^{12}.
\end{gather*}

Then $R(\alpha)>\hat{R}(\alpha)$ for $\alpha\ge0$. Let us form the sequence 
$\left\{\hat{R}(\alpha),\hat{R}'(\alpha),\hat{R}''(\alpha),\hat{R}'''(\alpha),\ldots,\hat{R}^{(12)}(\alpha)\right\}$
and count the number of sign variations in this sequence for $\alpha=0$ and $\alpha=0.5$. In both cases, this number
is equal to eleven. It follows from {\it Budan's theorem} that the polynomial $\hat{R}(\alpha)$ does not have zeros 
in the interval $\alpha\in[0,1/2]$. Moreover, $\hat{R}(0)>0$. Hence, $\hat{R}(\alpha)>0$ when $\alpha\in[0,1/2]$,
which implies that inequality \eqref{inl2} is valid. The lemma is proved.
\end{proof}

\begin{lemma}\label{l3}
The function $\ds\mu(t)=\frac{A(t)(1+B(t))}{B(t)}$ decreases when $t\in[0,\pi/2]$, $\alpha\in(0,1/2]$.
\end{lemma}
\begin{proof}
Let us find $\ds\mu'(t)=\frac1{B^2(t)}\left(A'(t)B(t)(1+B(t))-A(t)B'(t)\right)$. Taking into account (Lemma~\ref{l1})
that $\ds A'(t)=\frac{\sin(2-\alpha)t}{\sin t}B'(t)>0$, we then have
\begin{gather*}
\mu'(t)=\frac1{B^2(t)}\left(A'(t)B(t)(1+B(t))-A(t)B'(t)\right)=\\
=\frac{B'(t)\sin(2-\alpha)t}{B^2(t)\sin t}\left(B^2(t)-\frac{A(t)\sin t-B(t)\sin(2-\alpha)t}{\sin(2-\alpha)t}\right).
\end{gather*}
Note that $A(t)\sin t-B(t)\sin(2-\alpha)t=\sin\alpha t$. Then 
$$
\mu'(t)=\frac{B'(t)\sin(2-\alpha)t}{B^2(t)\sin t}\left(B^2(t)-\frac{\sin\alpha t}{\sin(2-\alpha)t}\right).
$$
The function $\ds\frac{\sin\alpha t}{\sin(2-\alpha)t}$ is increasing when $t\in[0,\pi/2]$, hence
$\ds\frac{\sin\alpha t}{\sin(2-\alpha)t}\ge\frac{\alpha}{2-\alpha}$. On the other hand, \lb
$B^2(t)\le B^2\left(\frac{\pi}2\right)=\left(\frac{\alpha}{2-\alpha}\right)^2<\frac{\alpha}{2-\alpha}$. 
Therefore, $\mu'(t)<0$, $t\in[0,\pi/2]$, $\alpha\in(0,1/2]$. The lemma is proved.
\end{proof}

\begin{lemma}\label{l4}
For $t^*=\frac{\pi}{3-\alpha}$, we have
$
\ds\mu(t^*)=\frac{A(t^*)(1+B(t^*))}{4B(t^*)}>1.
$
\end{lemma}
\begin{proof}
We will be reasoning in the same way as in the proof of Lemma~\ref{l2}. Denote $\ds z=\sin^2\frac{\pi}{3-\alpha}$
and write
$$
A(t^*)=a_1=-2+\frac8{3-\alpha}z,\quad B(t^*)=b_1=1-\frac{4(1-\alpha)}{3-\alpha}z.
$$
Then 
\begin{gather*}
A(t^*)(1+B(t^*))-4B(t^*)=\\
=\frac8{(3-\alpha)^2}\left(-4(1-\alpha)z^2+(3\alpha^2-14\alpha+15)z-\alpha^2+6\alpha-9\right)=
\frac{\Phi(\alpha,z)}{(3-\alpha)^2}.
\end{gather*}

Calculate
$$
\frac{\partial}{\partial z}\Phi(\alpha,z)=-8(1-\alpha)z+3\alpha^2-14\alpha+15>3\alpha^2-14\alpha+7>\frac34>0.
$$
Therefore, the function $\Phi(\alpha,z)$ increases in $z$. Estimate $z$ from below. Write
$$
\sin\frac{\pi}{3-\alpha}=\sin\left(\frac{\pi}3+\frac{\pi}3\frac{\alpha}{3-\alpha}\right)=
\frac{\sqrt3}2\cos\frac{\pi\alpha}{3(3-\alpha)}+\frac12\sin\frac{\pi\alpha}{3(3-\alpha)},
$$
whence
\begin{gather*}
\sin^2\frac{\pi}{3-\alpha}=\frac34-\frac12\sin^2\frac{\pi\alpha}{3(3-\alpha)}+
\frac{\sqrt3}4\sin\frac{2\pi\alpha}{3(3-\alpha)}>\\
>\frac34+\frac{\sqrt3}2\frac{\pi\alpha}{3(3-\alpha)}-\frac12\left(\frac{\pi\alpha}{3(3-\alpha)}\right)^2-
\frac{\sqrt3}3\left(\frac{\pi\alpha}{3(3-\alpha)}\right)^3=G_1(\alpha).
\end{gather*}
Then $\Phi(\alpha,z)>\Phi(\alpha,G_1(\alpha))=\frac{100\alpha R(\alpha)}{(3-\alpha)^6}$, where $R(\alpha)=\sum\limits_{j=0}^7 b_j\alpha^j$ is a polynomial of degree 7, and the coefficients 
$b_0$, $b_2$, $b_4$, $b_6$ are positive, and $b_1$, $b_3$, $b_5$, $b_7$ are negative. These coefficients can be
calculated with the required degree of accuracy and then estimated by replacing them by smaller ones, for which we
may leave, for instance, two decimal places. We will finally obtain 
$$
\hat{R}(\alpha)=3.43-15.44\alpha+27.03\alpha^2-23.47\alpha^3+10.51\alpha^4-2.40\alpha^5
+0.26\alpha^6-0.02\alpha^7.
$$

Then $R(\alpha)>\hat{R}(\alpha)$ for $\alpha\ge0$. Let us form the sequence 
$\left\{\hat{R}(\alpha),\hat{R}'(\alpha),\hat{R}''(\alpha),\hat{R}'''(\alpha),\ldots,\hat{R}^{(7)}(\alpha)\right\}$
and count the number of sign variations in this sequence for $\alpha=0$ and $\alpha=0.5$. In both cases, this number
is equal to seven. {\it Budan's theorem} implies that the polynomial $\hat{R}(\alpha)$ does not have zeros in the 
interval $\alpha\in[0,1/2]$. Moreover, $\hat{R}(0)>0$. Therefore, $\hat{R}(\alpha)>0$ when $\alpha\in[0,1/2]$, which
gives the validity of inequality \eqref{inl2}. The lemma is proved.
\end{proof}

\begin{lemma}\label{l5}
For $t\in[0,\pi/2]$, $\alpha\in(0,1/2]$ we have $A(t)\ge2\sqrt2 B(t)$.
\end{lemma}
\begin{proof}
Let us calculate
$\left(\frac{B(t)}{A(t)}\right)'=\frac{B'(t)}{A^2(t)\sin t}\left(A(t)\sin t-B(t)\sin(2-\alpha)t\right)=
\frac{B'(t)}{A^2(t)}\frac{\sin\alpha t}{\sin t}>0$ (see Lemma~\ref{l3}). Hence, $\frac{B(t)}{A(t)}\le\frac{B(\pi/2)}{A(\pi/2)}=\frac{\alpha}{2\sin(\alpha\pi/2)}\le\frac1{2\sqrt2}$ for $\alpha \in (0,\frac{1}{2}]$. The lemma is proved.
\end{proof}

\begin{lemma}\label{l6}
For $\alpha \in (0,1/8]$, we have $\ds\frac{(2-\alpha-\sin(\pi\alpha)/2)^2}{2(2-\alpha)(4-2\alpha-\sin(\pi\alpha)/2)}>\cos^2\frac{\pi}{3-\alpha}$.
\end{lemma}
\begin{proof}
It follows from Lemma~\ref{l2} that
\begin{gather*}
\cos^2\frac{\pi}{3-\alpha}<\frac14-\frac{\sqrt3}2\frac{\pi\alpha}{3(3-\alpha)}
+\frac12\left(\frac{\pi\alpha}{3(3-\alpha)}\right)^2+
\frac{\sqrt3}3\left(\frac{\pi\alpha}{3(3-\alpha)}\right)^3=G_1(\alpha),\\
\sin\frac{\pi\alpha}2<\frac{\pi\alpha}2-\frac{(\pi\alpha)^3}{48}+\frac{(\pi\alpha)^5}{3840}=G_2(\alpha).
\end{gather*}
The function $\ds\frac{(2-\alpha-z)^2}{2(2-\alpha)(4-2\alpha-z)}$ decreases in $z$. This means that 
$$
\frac{(2-\alpha-\sin(\pi\alpha)/2)^2}{2(2-\alpha)(4-2\alpha-\sin(\pi\alpha)/2)}>
\frac{(2-\alpha-G_2(\alpha))^2}{2(2-\alpha)(4-2\alpha-G_2(\alpha))}.
$$
It suffices to show that $\ds\frac{(2-\alpha-G_2(\alpha))^2}{2(2-\alpha)(4-2\alpha-G_2(\alpha))}>G_1(\alpha)$ when
$\alpha\in(0,1/8]$. Consider the polynomial
$$
\frac{(3-\alpha)^3}{\alpha}\left[(2-\alpha-G_2(\alpha))^2-2(2-\alpha)(4-2\alpha-G_2(\alpha))G_1(\alpha)\right]=
R(\alpha).
$$
$R(\alpha)=\sum\limits_{j=0}^{12}b_j\alpha^j$ is a polynomial of degree 12 with the positive coefficients
$b_0$, $b_2$, $b_4$, $b_5$, $b_7$, $b_8$, $b_{11}$ and the negative coefficients $b_1$, $b_3$, $b_6$, $b_9$,
$b_{10}$, $b_{12}$. These coefficients can be calculated with the required degree of accuracy and then estimated
by replacing them by smaller ones if leaving, for example, two decimal places. 
As a result, we obtain
\begin{gather*}
\hat{R}(\alpha)=3.35-37.79\alpha+108.54\alpha^2-132.33\alpha^3+63.91\alpha^4
+7.38\alpha^5-17.88\alpha^6+\\
+3.16\alpha^7+2.11\alpha^8
-0.76\alpha^9-0.07\alpha^{10}+0.05\alpha^{11}-0.01\alpha^{12}.
\end{gather*}

Then $R(\alpha)>\hat{R}(\alpha)$ when $\alpha\ge0$. Let us form the sequence 
$\left\{\hat{R}(\alpha),\hat{R}'(\alpha),\hat{R}''(\alpha),\hat{R}'''(\alpha),\ldots,\hat{R}^{(12)}(\alpha)\right\}$
and count the number of sign variations in this sequence for $\alpha=0$ and $\alpha=0.125$. In both cases, this number
is equal to nine. It follows from {\it Budan's theorem} that the polynomial $\hat{R}(\alpha)$ does not have zeros in the 
interval $\alpha\in[0,1/8]$. Moreover, $\hat{R}(0)>0$. Therefore, $\hat{R}(\alpha)>0$ when $\alpha\in[0,1/8]$, which 
implies that the statement of Lemma is true.   
\end{proof}

\section{Acknowledgements}

The authors would like to thank Larie Ward for her help in preparation of this manuscript.

\bibliographystyle{amsplain}
\bibliography{koebebib}

\end{document}